\documentclass[12pt,a4]{article}
\usepackage{amsmath,amssymb}
\usepackage{amsthm}
\usepackage{mathrsfs}
\usepackage[pdftex]{graphicx}
\usepackage{bm}
\usepackage[top=30truemm,bottom=30truemm,left=30truemm,right=30truemm]{geometry}
\usepackage[subrefformat=parens]{subcaption}
\usepackage{comment}

\DeclareMathOperator{\Length}{Length}
\DeclareMathOperator{\Vol}{Vol}

\renewcommand{\star}{{\rm star}\, }

\makeatletter

\@addtoreset{equation}{section}
\makeatother

\def\qed{\hfill $\Box$}

\def\<{\langle}
\def\>{\rangle}

\def\wt{\widetilde}

\def\sm{\setminus}

\def\vp{\varphi}

\def\C{\mathbb{C}}

\def\R{\mathbb{R}}

\def\Z{\mathbb{Z}}
\def\a{\alpha}

\theoremstyle{definition}
\newtheorem{thm}{Theorem}[section]

\newtheorem{prop}[thm]{Proposition}
\newtheorem{cor}[thm]{Corollary}
\newtheorem{lem}[thm]{Lemma}

\newtheorem{ex}[thm]{Example}
\newtheorem{rmk}{Remark}

\newtheorem{rmks}{Remarks}

\newtheorem{prf}{Proof}

\newtheorem{defn}[thm]{Definition}

\title{Stability problem of equilibrium discrete planar curves}
\author{Yoshiki Jikumaru
}
\date{}

\begin{document}
\maketitle

\begin{abstract}
In this paper, we study planar polygonal curves from the variational methods.
We show an unified interpretation of discrete curvatures and the Steiner-type formula by extracting the notion of the discrete curvature vector from the first variation of the length functional.
Moreover, we determine the equilibrium curves for the length functional under the area-constraint condition and study their stability.
\end{abstract}


\section{Introduction}

Discrete differential geometry is an active research field in connection with the smooth manifold theory and the  visualization on the computer.
In this field, from the theoretical point of view, one approach is based on the variational problems and the other is the integrability of equations.
We take the former approach and focus on discrete planar curves.
Surprisingly, such a discrete curve theory is not well-understood, therefore we try to develop them inspired by the work \cite{Pol2002}.
Moreover, we consider the stability problem for equilibrium discrete curves.
Stability problem for discrete objects have not developed well after the work by Polthier and Rossman \cite{PoRo2002}.

First we derive the first variation formula and extract the vector from the formula which should be called the ``discrete curvature vector'' in \S \ref{sec:FirstVariation}.
In \S \ref{sec:RelationDiscreteCurvature}, by using this curvature vector, we show an unifed interpretation which derives various kinds of discrete curvature notions introduced in \cite{Hat2019}, \cite{Hof2008}.
The important viewpoint here is that there is no natural notion of the line element on the vertices.
In \S \ref{sec:DiscreteCMCCurves}, we characterize equilibrium curves for the length functional under the area-constraint condition as regular polygons.
In \S \ref{sec:SteinerFormula}, we derive the Steiner-type formula for parallel curves by using the ``vertex normal'' constructed from the ``discrete curvature vector'' derived in \S \ref{sec:FirstVariation}.
In \S \ref{sec:SecondVariationCurves}, we will consider the stability problem of the regular polygons.
We can derive the second variation formula similar to the smooth case by decomposing the variation vector field to the ``normal'' and ``tangential'' directions as in \cite{BadoC1984}.
Moreover, we show the instability for the non-convex regular polygons by using the second variation formula for ``normal'' variations in the section \ref{sec:InstabilityOfRegularPolygons}.


\section{The first variation formula}
\label{sec:FirstVariation}

In this section we will consider the variation of discrete curves and extract the ``curvature vector'' from the first variation formula.
We can expect that this vector gives a notion of the curvature and normal at vertices.
Let us recall the basic definition of discrete planar curves.
\begin{defn}
Let $n$ be a non-negative integer.
A {\em standard (abstract) $n$-path} is a simplicial complex $G = (V, E)$ formed by
\begin{enumerate}
\item
$(n+1)$ abstract points : $V = \{ v_0, \ldots, v_n \}$.
\item
the set of $n$ edges $e_k = [v_k, v_{k+1}]$, $k = 0, \ldots, n-1$.
\end{enumerate}
An {\em standard abstract $n$-circle} is the union of a standard $n$-path and the ``final'' edge $e_n = [v_n, v_0]$.
A {\em discrete (planar) curve} is a geometric realization of a standard $(n-1)$-path or $(n-1)$-circle which is a map $X : V \to \R^2$ satisfies $l_k := |p_{k+1} - p_k| \neq 0$ for all $k = 0,1,\ldots,n-1$, where we denote $p_k := X(v_k)$.
We denote such a discrete curve as $\Gamma_h = \{ p_k \}_k$.
\end{defn}
For each oriented edge $e_k := [p_k, p_{k+1}]$ of $\Gamma_h$ we can assign an unit normal vector
\[
\nu_k 
:= R \left( \frac{p_{k+1} - p_k}{l_k} \right)
:= R \left( \frac{p_{k+1} - p_k}{|p_{k+1} - p_k|} \right), \quad k = 0, \ldots, n-1,
\]
where $R$ is the $\pi/2$-rotation or $-\pi/2$-rotation in $\R^2$.
It does not matter whichever we choose but we choose the same $R$ for all $k$.

For a discrete curve $\Gamma_h = \{ p_k \}_k$ with an orientation, the length of $\Gamma_h$ and the $2$-dimensional volume (area) bounded by $\Gamma_h$ are defined by
\[
L (\Gamma_h) := \sum_k l_k = \sum_k |p_{k+1} - p_k|, \quad
\Vol(\Gamma_h) := \frac{\,1\,}{2} \sum_k \< p_k, \nu_k \> |p_{k+1} - p_k|
\]
In these settings, we consider the following question:
\begin{center}
{\em What is the unit normal, curvature and line element at the vertices ?}
\end{center}
A realization of the dual graph may give an answer for this question, but we do not consider such a realization.
In order to approach the problem, we will extract the discrete curvature vector from the first variation of the length.
Therefore we first derive the first variation formula of the length functional.
In the discrete setting we consider the variation of vertices, that is, piecewise linear variations.
We consider a variation
\[
p_k(t) = p_k + t v_k +O(t^2), \quad k = 0, \ldots, n-1
\]
where ${}^t \vec{v} = ({}^t v_0, \ldots, {}^t v_{n-1}) \in \R^{2n}$ is the ``variation vector field''.
If $p_k$ is a boundary point of $\Gamma_h$, then we assume $v_k =0$.

We want to find the vector ${}^t \nabla L \in \R^{2n}$ satisfies
\[
\delta L
:= \frac{d}{dt}_{|t=0} L
= \< \vec{v}, \nabla L \>_{\R^{2n}}
= \sum_k \< v_k, \nabla_{p_k} L \>_{\R^2},
\]
where we write ${}^t \nabla L = ({}^t \nabla_{p_1} L, \ldots, {}^t \nabla_{p_n} L) \in \R^{2n}$.
Since this is just a direction derivative in $\R^{2n}$, the following proposition is immediately:
\begin{prop}
\label{prop:DiscreteFirstVariationFormulaCurve}
Let $\Gamma_h = \{ p_k \}_k$ be a discrete closed curve and $p_k$ be an interior vertex.
Then the gradient the length can be expressed in the following formula:
\begin{equation}
\label{eq:anisoEnergyGradCurve}
\nabla_{p_k} L 
= R (\nu_k - \nu_{k-1})
= - \frac{p_{k+1}-p_k}{l_k} + \frac{p_k - p_{k-1}}{l_{k-1}}.
\end{equation}
\end{prop}
By using this formula we have
\[
\delta L
= \sum_k \< \nabla_{p_k} L, v_k \>
= - \sum_k \left\< -\frac{\,1\,}{L_k} \nabla_{p_k} L, v_k \right\> L_k,
\]
where we inserted some auxilary function $L_k$ defined on the vertices.
This kind of observation is essentially remarked in the paper~\cite{ChGl2007}.
\begin{defn}[Discrete curvature vector]
For a positive function $L_k$ defined on the vertices, we call the vector
\begin{equation}
\label{eq:DiscreteCurvatureVector}
\wt{N}_k 
= - \frac{\,1\,}{L_k} \nabla_{p_k} L
= \frac{\,1\,}{L_k} \left( \frac{p_{k+1} - p_k}{l_k} - \frac{p_k - p_{k-1}}{l_{k-1}} \right)
\end{equation}
the {\em discrete curvature vector with respect to $L_k$}.
\end{defn}
\begin{rmks}
There are some reasons why we can regard the vector $\wt{N}_k$ as the curvature vector.
First, the vector $\wt{N}_k$ is independent of the choice of the unit normal, i.e., an intrinsic quantity.
Moreover, the second expression \eqref{eq:DiscreteCurvatureVector} can be regarded as a discretization of a part of the Frenet-Serret formula.
As we will remark below, suitable choices of the function $L_k$ derives various notions of discrete curvature defined in \cite{Hat2019}, \cite{Hof2008}.
Non-uniqueness of the function $L_k$ comes from the fact that there is no natural line element, i.e. the metric, at vertices.
\end{rmks}


\section{Relation with other notions of the discrete curvature}
\label{sec:RelationDiscreteCurvature}

In this section, non-uniqueness of the line element at the vertices gives various notions of the discrete curvature defined in \cite{Hat2019}, \cite{Hof2008}.

To describe the curvature notions, we have to define the angles at vertices.
We define (the absolute value of) the angle between $\nu_{k-1}$ and $\nu_k$ as $\theta_k$, i.e.,
\[
\cos \theta_k 
:= \< \nu_k, \nu_{k-1} \>
= \frac{ \< p_{k+1} - p_k, p_k - p_{k-1} \> }{|p_{k+1} - p_k| \cdot |p_k - p_{k-1}|}.
\]
We have to care about the signature of $\theta_k$.
Let $R_\theta$ be the $\theta$-rotation in $\R^2$.
We assign the signature depends on the choice of the rotation $R$ appeared in the definition of the edge normal:
\[
\sigma := 
\begin{cases}
+1 \quad &{\rm if} \quad R = R_{\pi/2} \\
-1 \quad &{\rm if} \quad R = R_{-\pi/2}
\end{cases}
\]
In this situation the signature of $\theta_k$ is determined by the equation $R_{\sigma \theta_k} (\nu_{k-1}) = \nu_k$.
We easily see that $\sum_k \theta_k = 2m \pi$ for some integer $m \in \Z$ for any closed curve.
We define the discrete curvature with respect to the choice of $L_k$ by the length of the discrete curvature vector with respect to $L_k$:
\begin{defn}
For a positive function $L_k$ defined on the vertices, we call the value
\[
\kappa(p_k) := \frac{2 \sin (\theta_k/2)}{L_k}
\]
the {\em discrete curvature with respect to} $L_k$.
\end{defn}
\begin{rmk}
If we use the expression $\nabla_{p_k} = R(\nu_k - \nu_{k-1})$, the discrete curvature with respect to $L_k$ can be written as $\kappa(p_k) = |\nu_k - \nu_{k-1}|/L_k$ up to the signature.
This can be regarded as a discretization of the curvature for a regular planar curve.
\end{rmk}
In the lecture note by Tim Hoffmann \cite{Hof2008}, three kinds of notions of the curvature for discrete curves are introduced:
\begin{enumerate}
\item
The curvature at vertices by using the vertex osculating circle method,
\item
The curvature at edges by using the edge osculating circle method,
\item
The curvature at vertices by using edge osculating circle for ``arclength parametrized'' curve.
\end{enumerate}
Moreover, Hatakeyama~\cite{Hat2019} also defined the curvature for discrete curves another way.
We will show that these curvature notions can be derived from our viewpoint.
\begin{prop}[The vertex osculating circle method~\cite{Hof2008}]
If we choose
\[
L_k 
= \frac{|p_{k+1} - p_{k-1}|}{2 \cos (\theta_k / 2)}
= \frac{|p_{k+1} - p_k + p_k - p_{k-1}|}{2 \cos (\theta_k / 2)},
\]
then the discrete curvature with respect to $L_k$ becomes
\[
\kappa(p_k) 
= \frac{2 \sin \theta_k}{|p_{k+1} - p_{k-1}|}
= \frac{2 \sin \theta_k}{|p_{k+1} - p_k + p_k - p_{k-1}|}
\]
and this value coincides with the curvature based on the vertex osculating circle method.
\end{prop}
\begin{prop}[For the arclength parametrized curves~\cite{Hof2008}]
Assume $l_k = l_{k-1} = l_0$.
Then if we choose
\[
L_k = l_0 \cos \frac{\theta_k}{2} = \frac{l_k + l_{k-1}}{2} \cdot \cos \frac{\theta_k}{2},
\]
then the discrete curvature with respect to $L_k$ becomes
\[
\kappa(p_k) = \frac{2}{l_0} \tan \frac{\theta_k}{2}
\]
and this value coincides with the curvature of arclength parametrized curve.
\end{prop}
In the paper~\cite{Hat2019}, the discrete curvature at the vertex is defined as
\[
\kappa(p_k) 
:= \frac{1}{|p_k - p_{k-1}|} \left| \frac{p_{k+1} - p_k}{|p_{k+1} - p_k|} - \frac{p_k - p_{k-1}}{|p_k - p_{k-1}|}\right|
= - \frac{|\nabla_{p_k} L|}{|p_k - p_{k-1}|}.
\]
Then we immediately have the following result:
\begin{prop}
If we choose $L_k = l_{k-1} = |p_k - p_{k-1}|$, then the discrete curvature with respect to $L_k$ coincides with the discrete curvature defined in \cite{Hat2019}.
\end{prop}
Before considering the edge osculating circle method, we shall modify the first variation formula from the vertex-based expression to the edge-based expression.
If we put $v_k = (w_k + w_{k-1})/2$, then we have
\begin{align*}
\delta L
&= \frac{1}{2} \sum_k \< \nabla_{p_k} L, w_k + w_{k-1} \> \\
&= \frac{1}{2} \sum_k \< \nabla_{p_k} L + \nabla_{p_{k+1}} L, w_k \>
= - \sum_k \left\< \frac{R(\nu_{k-1} - \nu_{k+1})}{2L_k'}, w_k \right\> L_k'
\end{align*}
where $L_k'$ is some auxiliary function.
As in the vertex case, we call the value
\[
\kappa(e_k) := \frac{1}{L_k'} \cdot \sin \frac{\theta_k + \theta_{k+1}}{2}
\]
the {\em discrete curvature at the edge $e_k = [p_k, p_{k+1}]$ with respect to $L_k'$}.
\begin{prop}[The edge osculating circle method~\cite{Hof2008}]
If we choose
\[
L_k' 
= l_k \cos \frac{\theta_k}{2} \cos \frac{\theta_{k+1}}{2}
= |p_{k+1} - p_k| \cos \frac{\theta_k}{2} \cos \frac{\theta_{k+1}}{2},
\]
then the discrete curvature with respect to $L_k'$ becomes
\[
\kappa(e_k) = \frac{\tan (\theta_k/2) + \tan (\theta_{k+1}/2)}{|p_{k+1} - p_k|}
\]
and this value coincides with the curvature based on the edge osculating circle method.
\end{prop}
\begin{rmk}
To define the discrete curvature, we have to choose $L_k$ (respectively $L_k'$) properly.
That means if $l_k$, $l_{k-1} \to ds$ and $\theta_k \to 0$, then $L_k$ (respectively $L_k'$) must converge to $ds$, i.e., $L_k$ must be a ``good'' candidate for a discrete line element.
We can check that $L_k$ and $L_k'$ satisfy this condition in the above examples.
\end{rmk}
\begin{rmk}[``No free lunch'' for the discrete Laplacian, cf. \cite{WaMaKaGr2007}]
We consider these kinds of ``no free lunch'' story for the discrete Laplacian which will be used in the second variation formula.
On a discrete curve $\Gamma_h$ with the vertex set $V$, we consider a function $\psi : V \to \R$.
Then the {\em gradient} and the {\em Laplacian} of $\psi$ can be defined as
\[
\nabla \psi_k := \frac{\psi_{k+1} - \psi_k}{l_k},\quad
\Delta \psi_k := \frac{1}{L_k} (\nabla \psi_k - \nabla \psi_{k-1}) 
=\frac{1}{L_k} \left( \frac{\psi_{k+1} - \psi_k}{l_k} - \frac{\psi_k - \psi_{k-1}}{l_{k-1}} \right),
\]
where we denote $\psi_k := \psi(k)$.
Note that the gradient is the ``edge-based operator'' but the Laplacian is the ``vertex-based'' operator.
In addition, the discrete curvature vector $\wt{N}_k$ with respect to $L_k$ can be written as $\Delta p_k$.

From another point of view, if we define the Dirichlet energy of $\psi$ as
\[
E_h(\psi) 
:= \frac{1}{2} \sum_k |\nabla \psi_k|^2 \, l_k
= \frac{1}{2} \sum_k \frac{|\psi_{k+1} - \psi_k|^2}{l_k},
\]
then the first variation of the energy becomes
\[
\delta E_h (\psi)
= \sum_k \frac{\< \psi_{k+1} - \psi_k, \vp_{k+1} - \vp_k \>}{l_k}
= \sum_k \< \nabla \psi_{k-1} - \nabla \psi_k, \vp_k \> 
= - \sum_k \< \Delta \psi_k, \vp_k \> L_k,
\]
where we take the variation of $\psi$ as $\psi_k (t) = \psi_k + t \vp_k + O(t^2)$.
Therefore $\delta E_h (\psi) = 0$ if and only if $\Delta \psi_k = 0$.
Note that the condition $\Delta \psi_k = 0$ is independent of the choice of $L_k$.

As in the curvature case, the Laplacian can be changed since there is no natural ``line element divisor $L_k$''.
However, with another function $\vp : V \to \R$, we still have the following properties since the quanties $\Delta \psi_k L_k$ are independent of $L_k$:
\begin{enumerate}
\item
If $\psi$ is constant, then $\Delta \psi = 0$.
\item
The condition $\Delta \psi = 0$ is independent of the choice of $L_k$ and in this case we have the mean value property:
\[
\psi_k = \frac{l_{k-1}}{l_k + l_{k-1}} \psi_{k+1} + \frac{l_k}{l_k + l_{k-1}} \psi_{k-1}. 
\]
\item $L^2$ symmetric property:
\[
\sum_k \psi_k \cdot \Delta \vp_k \cdot L_k = \sum_k \Delta \psi_k \cdot \vp_k \cdot L_k.
\]
Note that the summation is vertex-based.
\item Integration by parts:
\[
- \sum_k \psi_k \cdot \Delta \vp_k \cdot L_k = \sum_k \nabla \psi_k \cdot \nabla \vp_k \cdot l_k.
\]
Note that the right hand side is the edge-based summation but the left hand side is the vertex-based summation.
As a corollary, the operator $-\Delta$ is positive semi-definite.
\end{enumerate}
\end{rmk}


\section{Equilibrium curves of the length functional}
\label{sec:DiscreteCMCCurves}

In the previous section, we showed that non-uniqueness of the line element at vertices gives various discrete curvature notions.
However, the equilibrium curves for the length functional under the area-constraint condition should be characterized as some ``constant curvature'' objects by virture of the smooth case.
In this section, we show that such equilibrium curves can be characterized as regular polygons and that they are certainly regarded as ``constant curvature'' objects.

By a direct calculation we have the following result.
\begin{lem}
For any vertex $p_k$ of $\Gamma_h$ the gradient of the area $\nabla_{p_k} \Vol$ is given by
\[
\nabla_{p_k} \Vol = \frac{1}{2} R (p_{k+1} - p_{k-1}).
\]
\end{lem}
\begin{rmk}[Another ``no free lunch'' story]
We can modify the first variation formula of the volume as follows:
\[
\delta \Vol
= \frac{1}{2} \sum_k \< R(p_{k+1} - p_k + p_k - p_{k-1}), v_k \>
= \sum_k \left\< \frac{l_k \nu_k + l_{k-1} \nu_{k-1}}{2 L_k}, v_k \right\> L_k.
\]
At a glance, it seems like a natural to choose $2L_k = l_k + l_{k-1} = \Length( \star (p))$ and this is also frequently used as a ``vertex normal'' (a weighted sum of the edge normals):
\[
N_k^V := \frac{l_k \nu_k + l_{k-1} \nu_{k-1}}{l_k + l_{k-1}}.
\]
In addition, we have
\[
-\nabla_{p_k} L
= R(\nu_{k-1} - \nu_k)
= \frac{\sin \theta_k}{1 + \cos \theta_k} (\nu_k + \nu_{k-1})
= 2 \tan \frac{\theta_k}{2} \cdot \frac{\nu_k + \nu_{k-1}}{2}.
\]
by a simple calculation.
Therefore, unless the curve is arclength parameterized, there are (at least) two choices of the ``vertex normal'' from the variational viewpoint: using the length gradient (length descent direction) or using the volume gradient (volume descent direction).
This suggests that, in contrast to the smooth case, we have to choose the ``prefered'' vertex normal according to the energy in question.
\qed
\end{rmk}
\begin{ex}[Regular polygons]
\label{ex:RegularPolygons}
Let us take a discrete curve $\Gamma_h^{m,n} = \{ p_k \}_k$ as in the following way (including non-convex regular $n$-gon with radius $a$):
\[
p_k = a \exp (2 \pi \sqrt{-1} m k/n),
	\quad k=0, \ldots, n-1, \quad 1 \leq m \leq n-1,
\]
where we assume that $m/n \neq 1/2$.
In particular, we sometimes call the curve $\Gamma_h^{1,n}$ as a convex regular $n$-gon.
Note that $\Gamma_h^{n-1, n}$ is also convex but it has an opposite unit normal with $\Gamma_h^{1,n}$ (usually we assume that $\Gamma_h^{1,n}$ has the outward-pointing unit normal).
\begin{figure}[htbp]
\begin{tabular}{c}
	\begin{minipage}{0.25\hsize}
		\centering
		\includegraphics[height=2.2cm]{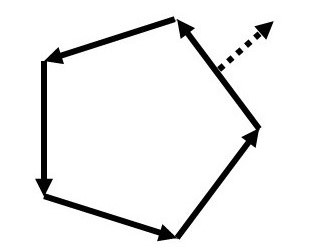}
		\subcaption{$\kappa = -1/\cos(\pi/5)$}
	\end{minipage}
	\begin{minipage}{0.25\hsize}
		\centering
		\includegraphics[height=2.2cm]{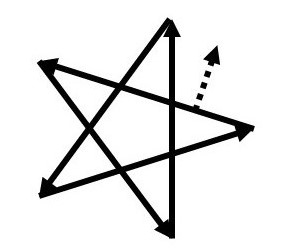}
		\subcaption{$\kappa = -1/\cos(2\pi/5)$}
	\end{minipage}
	\begin{minipage}{0.25\hsize}
		\centering
		\includegraphics[height=2.2cm]{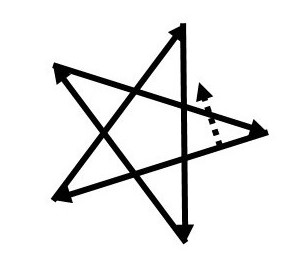}
		\subcaption{$\kappa = -1/\cos(3\pi/5)$}
	\end{minipage}
	\begin{minipage}{0.25\hsize}
		\centering
		\includegraphics[height=2.2cm]{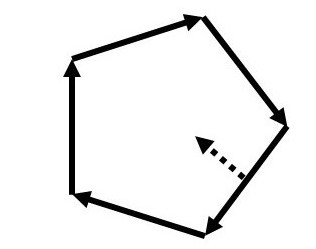}
		\subcaption{$\kappa = -1/\cos(4\pi/5)$}
	\end{minipage}
\end{tabular}
\caption{Convex and non-convex regular $5$-gons with radii $a=1$.}
\end{figure}
Then the curve $\Gamma_h^{m,n}$ is a critical point of the functional $\Length + \kappa \Vol$ with $\kappa = - 1/ (a \cos (m \pi /n))$.
This value is the reciprocal of the radius of the inscribed circle of the polygon (up to the signature).
We sometimes say that a convex regular $n$-gon with radius $a$ (and outward-pointing unit normal) has constant curvature $\kappa_n = -1/(a \cos (\pi/n))$.
Note that $\cos(\pi m/n) = -\cos (\pi (n-m)/n)$ and $\kappa_n \to -1/a$ when $n \to \infty$.
\end{ex}
We will show that these regular polygons are the only equilibrium curves for the functional $L + \kappa \Vol$.
\begin{thm}
Let $\Gamma_h = \{ p_k \}_{k=1}^n$ be a closed discrete curve and take $\kappa \in \R \sm \{ 0 \}$.
Then the following two conditions are equivalent:
\begin{enumerate}
\item
$\Gamma_h$ is an equilibrium curve of the functional $L + \kappa \Vol$.
\item
There exist numbers $l_0$ and $\theta_0$ such that $l_k \equiv l_0$ and $\theta_k \equiv \theta_0$ satisfying $\kappa l_0 = 2 \tan (\theta_0 /2)$, i.e., $\Gamma_h$ must be a regular polygon.
\end{enumerate}
\end{thm}
\begin{prf}
We put 
\[
A_k := (\nu_k - \nu_{k-1}) + \frac{\kappa}{2} (p_{k+1} - p_{k-1}), \quad k=1, \ldots, n. 
\]
Then the discrete curve $\Gamma_h$ is a critical point of the functional $\Length + \kappa \Vol$ if and only if $A_k = 0$ for all $k$.
By a simple calculation we have
\begin{align}
\label{eq:critical1}
\< A_k, \nu_{k-1} \> &= \sin \theta_k \left( \frac{\kappa l_k}{2} - \tan \frac{\theta_k}{2} \right), \\
\label{eq:critical2}
\< A_k, \nu_k \> &= \sin \theta_k \left( \tan \frac{\theta_k}{2} - \frac{\kappa l_{k-1}}{2} \right), \\
\label{eq:critical3}
\< A_{k+1}, \nu_{k+1} \> &= \sin \theta_{k+1} \left( \tan \frac{\theta_{k+1}}{2} - \frac{\kappa l_k}{2} \right).
\end{align}
For the necessity, that is, if we assume $A_k = 0$ for all $k$, then it follows from \eqref{eq:critical1} and \eqref{eq:critical2} that $\kappa l_k = 2 \tan (\theta_k /2) = \kappa l_{k-1}$.
And it also follows from \eqref{eq:critical2}, \eqref{eq:critical3} and using $l_k = l_{k-1}$ that
\[
\tan \frac{\theta_k}{2} = \frac{\kappa l_k}{2} = \frac{\kappa l_{k-1}}{2} = \tan \frac{\theta_{k-1}}{2}.
\]

To prove the sufficiency, since $\nu_k$ and $p_k - p_{k-1}$ forms a basis of $\R^2$ and $\< A_k, \nu_k \> = 0$ by using \eqref{eq:critical1} and the assumption, all we have to prove is $\< A_k, p_k - p_{k-1} \> = 0$ for all $k$.
By using the assumption $l_k = l_{k-1} = l_0$ and $\theta_k = \theta_0$, we have
\begin{align*}
\left\< A_k, \frac{p_k - p_{k-1}}{l_k} \right\>
&= - \sin \theta_k +\frac{\kappa}{2} (l_k \cos \theta_k + l_{k-1}) \\ 
&= (1 + \cos \theta_0) \left( \frac{\kappa l_0}{2} - \frac{\sin \theta_0}{1+\cos \theta_0} \right)
= (1 + \cos \theta_0) \left( \frac{\kappa l_0}{2} - \tan \frac{\theta_0}{2} \right)
= 0.
\end{align*}
This shows $A_k=0$ and proves the statement.
\qed
\end{prf}
\begin{rmk}
The equilibrium condition $A_k = \nu_k - \nu_{k-1} + (\kappa/2) (p_{k+1} - p_{k-1}) = 0$ is equivalent to the condition $\nu_k + (\kappa/2) (p_{k+1} + p_k) \equiv c$ for some constant vector $c \in \R^2$.
The latter condition can be considered as a conservation law for the Euler-Lagrange equation $A_k = 0$.
Since the vector $c \in \R^2$ is just a translation of the curve, we can put $c=0$ and in this case we have $(p_{k+1}+p_k)/2 = - \nu_k / \kappa$.
Therefore, the edge midpoints must be tangent to the unit circle.
\end{rmk}
We found that equilibrium closed curves of the functional $L + \kappa \Vol$ must satisfy $l_k \equiv l_0$, i.e., they must have ``good coordinates (arclength parameter)''.
If we note that we can define the curvature at vertices for an arclength parametrized curve, the previous result can be restated as follows:
\begin{cor}
Let $\Gamma_h$ be an arclength parametrized discrete closed curve, i.e., $l_k \equiv l_0$, and take $\kappa \in \R \sm \{ 0 \}$.
Then the following two conditions are equivalent:
\begin{enumerate}
\item
$\Gamma_h$ is an equilibrium curve of the functional $\Length + \kappa \Vol$.
\item
The discrete curvature $(2/l_0) \tan (\theta_k /2)$ is constant $\kappa$.
\end{enumerate}
\end{cor}


\section{Parallel curves and Steiner-type formula}
\label{sec:SteinerFormula}

In this section we will derive the discrete version of the Steiner-type formula in order to show an effectiveness of the vertex normal which we will define.
The following type of Steiner formula is essentially appeared in some papers, e.g. \cite{BaPoWa2010}, \cite{CrWa2017}.
Although they try to {\em find} the curvature notion from the Steiner-type formula, we will {\em derive} the Steiner-type formula by using our vertex normal and connect with the well-known curvature notion.

Let $\Gamma_h = \{ p_k \}_k$ be a discrete curve and take an interior vertex $p_k$.
We can expect that if we normalize the discrete curvature vector, then we have the ``vertex normal''.
Recall that the length of the length gradient $\nabla_{p_k} L$ can be computed as $2 \sin (\theta_k/2)$ up to the signature.
In the discrete case, we should consider another factor $2 \sin (\theta_k/2) \cos (\theta_k/2) = \sin \theta_k$ and put
\begin{equation}
\label{eq:VertexNormal}
N_k := -\frac{\nabla_{p_k} \Length}{\sin \theta_k} 
= \frac{R (\nu_{k-1} - \nu_k)}{\sin \theta_k} 
= \frac{1}{1 + \cos \theta_k} (\nu_k + \nu_{k-1}),
\end{equation}
then we shall call the vector $N_k$ as the {\em vertex normal} at the vertex $p_k$.
The second expression in \eqref{eq:VertexNormal} allows us to define the vertex normal even if $\theta_k=0$.
Then we consider the following deformation of the curve:
\[
p_k(t) = p_k + t N_k, \quad k=1,\ldots, n.
\]
\begin{lem}
We have $\< p_{k+1}(t) - p_k(t), \nu_k \> = 0$, therefore we call this deformation {\em parallel curves}.
\end{lem}
\begin{prf}
This is a direct calculation.
\begin{align*}
\< p_{k+1}(t) - p_k(t), \nu_k \>
&= \< (p_{k+1} - p_k) + t (N_{k+1} - N_k), \nu_k \> \\
&= t \left( \frac{\< \nu_{k+1} + \nu_k, \nu_k \>}{1 + \cos \theta_{k+1}}  
	- \frac{\< \nu_k + \nu_{k-1}), \nu_k \>}{1 + \cos \theta_k} \right) \\
&= t \left( \frac{1 + \cos \theta_{k+1}}{1 + \cos \theta_{k+1}}  
	- \frac{1 + \cos \theta_k}{1 + \cos \theta_k} \right) 
= 0.
\end{align*}
\qed
\end{prf}
\begin{thm}[Discrete Steriner-type formula]
For parallel curves $\{ p_k(t) \}_k$, we have
\[
|p_{k+1} (t) - p_k(t)| = |p_{k+1} - p_k| (1 - t \cdot \kappa(e_k)),
\]
where $\kappa(e_k)$ is the discrete curvature based on the edge osculating circle method~\cite{Hof2008}:
\[
\kappa(e_k) = \frac{\tan (\theta_k/2) + \tan (\theta_{k+1}/2)}{|p_{k+1} - p_k|}.
\]
\end{thm}
Before giving the proof, we give an intuitive explanation for a special case.
\begin{figure}[htbp]
	\centering
	\includegraphics[height=3cm]{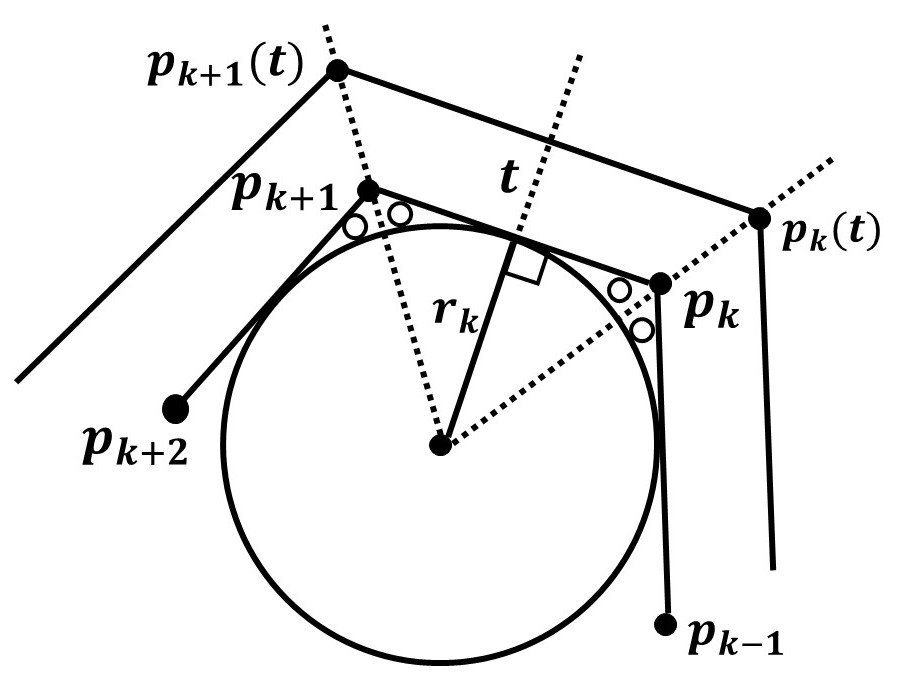}
	\caption{Parallel curves}
	\label{Figure:ParallelCurve}
\end{figure}
In the Figure~\ref{Figure:ParallelCurve}, the similarity ratio of the triangles gives
\begin{align*}
|p_{k+1}(t) - p_k(t)| : |p_{k+1} - p_k| = (-r_k + t) : -r_k 
\iff |p_{k+1}(t) - p_k(t)| = |p_{k+1} - p_k| (1 - t \cdot \kappa(e_k)).
\end{align*}
Note that the signature of the curvature radius $r_k$ is negative in the figure. 
\begin{prf}
By a calculation we have
\begin{align*}
\left\< \frac{p_{k+1} -p_k}{|p_{k+1} -p_k|}, N_{k+1} - N_k \right\>
&= \left\< - R \nu_k, \frac{\nu_k + \nu_{k+1}}{1+\cos \theta_{k+1}} - \frac{\nu_{k-1} + \nu_k}{1 + \cos \theta_k} \right\> \\
&= - \left( \frac{\sin \theta_{k+1}}{1+\cos \theta_{k+1}} + \frac{\sin \theta_k}{1+\cos \theta_k} \right)
= - \left( \tan \frac{\theta_{k+1}}{2} + \tan \frac{\theta_k}{2} \right),
\end{align*}
\begin{align*}
|N_{k+1} - N_k|^2
&= \frac{1}{\cos^2 (\theta_{k+1}/2)} + \frac{1}{\cos^2 (\theta_k/2)} -2 \< N_k, N_{k+1} \> \\
&= 2 + \tan^2 \frac{\theta_{k+1}}{2} + \tan^2 \frac{\theta_k}{2}
	-2 \cdot \frac{1 + \cos \theta_k + \cos \theta_{k+1} + \cos (\theta_k + \theta_{k+1})}{(1 + \cos \theta_k) (1 + \cos \theta_{k+1})} \\
&= \left( \tan \frac{\theta_k}{2} + \tan \frac{\theta_{k+1}}{2} \right)^2.
\end{align*}
Therefore we conclude
\begin{align*}
&\qquad |p_{k+1}(t) - p_k(t)|^2 \\
&= |p_{k+1} - p_k|^2 + 2t \< p_{k+1} -p_k, N_{k+1} -N_k  \> + t^2 |N_{k+1} - N_k|^2 \\
&= |p_{k+1} - p_k|^2 - 2t |p_{k+1} - p_k| \left( \tan \frac{\theta_k}{2} + \tan \frac{\theta_{k+1}}{2} \right)
	+ t^2 \left( \tan \frac{\theta_k}{2} + \tan \frac{\theta_{k+1}}{2} \right)^2 \\
&= |p_{k+1} - p_k|^2 \left( 1 - t \cdot \frac{\tan (\theta_k/2) + \tan (\theta_{k+1}/2)}{|p_{k+1} - p_k|} \right)^2.
\end{align*}
\qed
\end{prf}
\begin{rmk}
This formula itself is already appeared some papers to {\em find} the discrete curvature at the vertices by using the offsets, see \cite{BaPoWa2010}, \cite{CrWa2017}.
If we take the offset of the edges, then we have the following (at least) three possiblities.
\begin{figure}[htbp]
\begin{tabular}{c}
	\begin{minipage}{0.33\hsize}
		\centering
		\includegraphics[height=2cm]{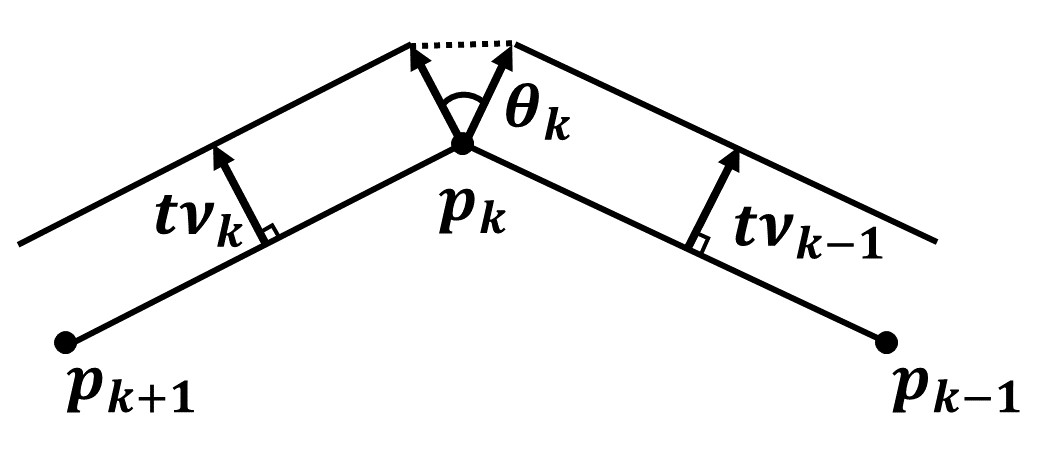}
		\subcaption{$\Gamma_{h, t}^{(1)}$, connect by a segment}
	\end{minipage}
	\begin{minipage}{0.33\hsize}
		\centering
		\includegraphics[height=2cm]{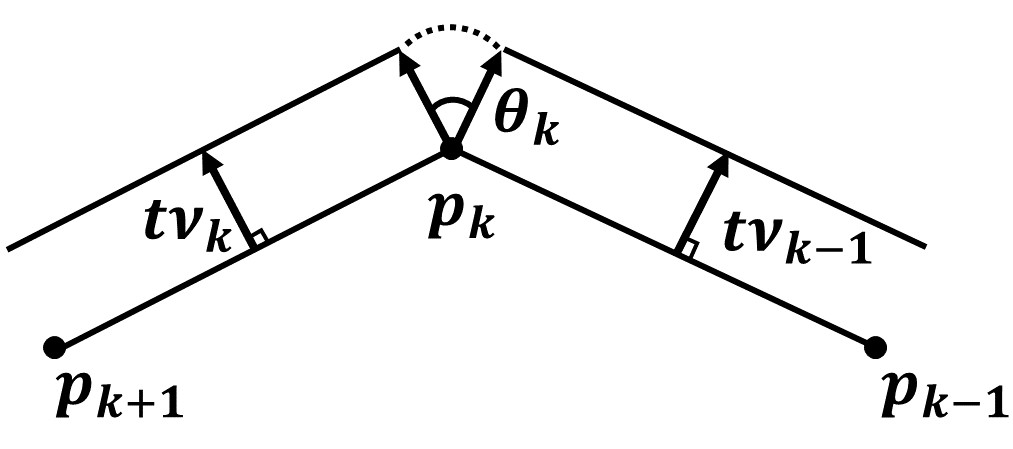}
		\subcaption{$\Gamma_{h, t}^{(2)}$, connect by an arc}
	\end{minipage}
	\begin{minipage}{0.33\hsize}
		\centering
		\includegraphics[height=2cm]{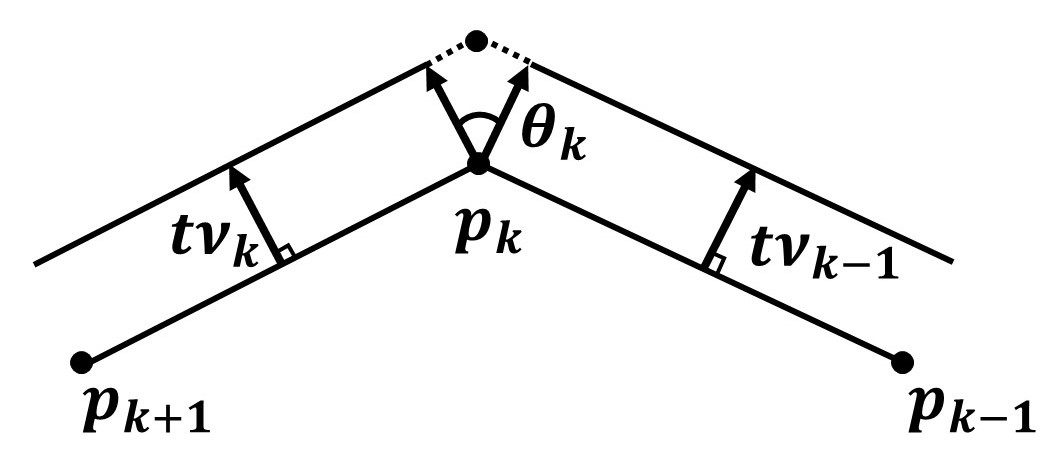}
		\subcaption{$\Gamma_{h, t}^{(3)}$, connect by a wedge}
	\end{minipage}
\end{tabular}
\caption{Three possibilities to construct a new curve}
\end{figure}
By computing the length of dotted curves in the figure, we can write the total length of each offsets as follows:
\begin{align*}
\Length (\Gamma_{h, t}^{(1)}) 
&= \Length (\Gamma_h) - t \sum_k 2 \sin \frac{\theta_k}{2}, \quad
\Length(\Gamma_{h, t}^{(2)}) 
= \Length(\Gamma_h) - t \sum_k \theta_k, \\
\Length (\Gamma_{h, t}^{(3)}) 
&= \Length(\Gamma_h) - t \sum_k 2 \tan \frac{\theta_k}{2}.
\end{align*}
Note the signature of the angles.
The second curve $\Gamma_h^{(2)}$ is nothing but the normal cone method (or the boundary of the Minkowski sum with the disk) known in the convex geometry.
However, the only possible way to unchange the number of the vertices during the offset procedure is the third one, and by modifying the third formula gives our Steiner-type formula:
\begin{align*}
\Length(\Gamma_{h, t}^{(3)})
&= \sum_k (l_k - 2t \tan (\theta_k / 2)) \\
&= \sum_k (l_k - t (\tan (\theta_k / 2) + \tan (\theta_{k+1} /2)))
= \sum_k l_k (1 - t \kappa(e_k)),
\end{align*}
where we put $\kappa(e_k) = (\tan (\theta_k / 2) + \tan (\theta_{k+1} /2)) / l_k$.
\end{rmk}
\begin{rmk}[Discrete Frenet-Serret formula]
In \S \ref{sec:FirstVariation}, we remarked that the relation
\[
\wt{N}_k = \frac{\,1\,}{L_k} (t_k - t_{k-1}), \quad t_k = \frac{p_{k+1} - p_k}{|p_{k+1} - p_k|} 
\]
can be considered as a part of the Frenet-Serret formula.
On the other hand, by some calculation we have
\[
\frac{1}{|p_{k+1} - p_k|} (N_{k+1} - N_k)
= - \frac{\tan (\theta_k/2) + \tan (\theta_{k+1}/2)}{|p_{k+1} - p_k|} t_k
= - \kappa (e_k) t_k.
\]
Note that the former formula is the formula on the vertex $p_k$ while the latter formula is the formula on the edge $[p_k, p_{k+1}]$.
\end{rmk}


\section{The second variation formula}
\label{sec:SecondVariationCurves}

In this section we consider the second variation formula of the length functional.
We will follow the argument developed in \cite{PoRo2002}.

Let $\Gamma_h = \{ p_k \}_k$ be an equilibrium closed curve for the functional $L + \kappa \Vol$.
We say a variation is {\em admissible} (or {\em permissible}) if the variation is volume-preserving and fixes the boundary.
We recall the first variation formula of the length and the $2$-dimensional volume:
\begin{align*}
\frac{d}{dt} L &= \sum_k \< \nabla_{p_k} L, p_k' \>, 
	\quad \nabla_{p_k} L = R ( \nu_k - \nu_{k-1} ), \\
\frac{d}{dt} \Vol &= \sum_k \< \nabla_{p_k} \Vol, p_k' \>, \quad \nabla_{p_k} \Vol = \frac{1}{2} R(p_{k+1} - p_{k-1}).
\end{align*}
Note that if the variation is admissible, then we have
\[
0 = \delta^2 \Vol = \sum_k \< \delta (\nabla_{p_k} \Vol), \delta p_k \> + \< \nabla_{p_k} \Vol, \delta^2 p_k \>.
\]
\begin{lem}
Let $\Gamma_h = \{ p_k \}_k$ be an equilibrium closed curve for the functional $L + \kappa \Vol$ and $p_k(t) = p_k + t v_k + O(t^2)$ be an admissible variation.
Then we have
\[
\delta^2 L 
:= \frac{d^2}{d t^2}_{|t=0} L
= \sum_k \< \delta (\nabla_{p_k} L + \kappa \nabla_{p_k} \Vol), v_k \>.
\]
\end{lem}
\begin{defn}[Stability of discrete curves]
Let $\Gamma_h = \{ p_k \}_k$ be a closed equilibrium curve for the functional $L + \kappa \Vol$.
Then $\Gamma_h$ is said to be {\em stable} if  $\delta^2 L \geq 0$ for any admissible variation.
\end{defn}
We introduce the matrix $Q^L$ and $Q^V$ as follows:
\[
{}^t \vec{v} Q^L \vec{v} = \sum_k \< \delta \nabla_{p_k} L, v_k \>, \quad
{}^t \vec{v} Q^V \vec{v} = \sum_k \< \delta \nabla_{p_k} \Vol, v_k \>.
\]
Then we can write $\delta^2 L = {}^t \vec{v} (Q^L + \kappa Q^V) \vec{v}$.
\begin{lem}
${}^t \vec{v} Q^V \vec{v} = \sum_k \< v_k, R v_{k+1} \>$.
\end{lem}
\begin{prf}
The proof follows from the direct computation:
\begin{align*}
{}^t \vec{v} Q^V \vec{v}
&= \sum_k \< \delta \nabla_{p_k} \Vol , v_k \>
= \frac{1}{2} \sum_k \< \delta R(p_{k+1} - p_{k-1}), v_k \> \\
&= \frac{1}{2} \sum_k \< R(v_{k+1} - v_{k-1}), v_k \>
= \sum_k \< v_k, R v_{k+1} \>.
\end{align*}
\qed
\end{prf}
\begin{prop}[Second variation formula for the length functional]
\[
Q ^L
= \sum_k \frac{1}{l_k} (|v_{k+1} - v_k|^2 - \< v_{k+1} - v_k, R \nu_k \>^2)
= \sum_k ( |\nabla v_k|^2 - \< \nabla v_k, R \nu_k \>^2) l_k,
\]
therefore we have the following second variation formula for the length functional:
\begin{equation}
\label{eq:2ndvarcurve}
\delta^2 L
= \sum_k ( |\nabla v_k|^2 - \< \nabla v_k, R \nu_k \>^2) l_k + \kappa \< v_k, R v_{k+1} \>.
\end{equation}
\end{prop}
\begin{prf}
By using the fact $l_k \delta \nu_k = R(v_{k+1} - v_k) - \< R(v_{k+1} - v_k), \nu_k \> \nu_k$, we have
\begin{align*}
Q^L
&= \sum_k \< \delta \nu_k, R(v_{k+1} - v_k) \> 
= \sum_k \< R(\nabla v_k) - \< R(\nabla v_k), \nu_k \> \nu_k, R(\nabla v_k) \> l_k \\
&= \sum_k ( |\nabla v_k|^2 - \< R(\nabla v_k), \nu_k \>^2) l_k
= \sum_k ( |\nabla v_k|^2 - \< \nabla v_k, R \nu_k \>^2) l_k.
\end{align*}
\qed
\end{prf}
In the section of the Steiner-type formula, we used the vector
\[
N_k := \frac{R(\nu_{k-1} - \nu_k)}{\sin \theta_k} 
= \frac{1}{1 + \cos \theta_k} (\nu_k + \nu_{k-1})
\]
as the ``normal vector'' at the vertex $p_k$.
If we define the ``tangent vector'' $T_k$ as $T_k = -RN_k$, then we can decompose the variation vector $v_k$ as
\[
v_k = \psi_k N_k + \eta_k T_k = \psi_k N_k - \eta_k RN_k,
\]
where $\psi, \eta : V \to \R$ is some functions on the vertices.
If $\eta_k = 0$ for all $k$, we call the variation the {\em normal variation}.
In the following we will use this notation.
\begin{lem}
\label{lem:AreaPreserving}
The first variation formula of the volume can be written as
\[
\delta \Vol
= \frac{1}{2} \sum_k \left( \psi_k (l_k + l_{k-1}) + \eta_k (l_k - l_{k-1}) \tan \frac{\theta_k}{2} \right).
\]
In particular, if the curve $\Gamma_h  = \{ p_k \}_k$ satisfies $l_k \equiv l_0$, then a variation $p_k(t) = p_k + t (\psi_k N_k + \eta_k T_k) + O(t^2)$ is volume-preserving if $\sum_k \psi_k = 0$.
\end{lem}
\begin{prf}
Using the first variation formula, we have
\begin{align*}
\delta \Vol
&= \frac{1}{2} \sum_k \< R(p_{k+1} - p_{k-1}), v_k \>
= \frac{1}{2} \sum_k \< R(p_{k+1} - p_{k-1}), \psi_k N_k + \eta_k T_k \> \\
&= \frac{1}{2} \sum_k 
	\left[ \psi_k l_k \frac{\< \nu_k, R \nu_{k-1} \>}{\sin \theta_k} 
		+ \psi_k l_{k-1} \frac{\< \nu_{k-1}, -R \nu_k \>}{\sin \theta_k}
		+ \frac{\eta_k}{\sin \theta_k} (l_k - l_{k-1}) (1- \cos \theta_k) \right] \\
&= \frac{1}{2} \sum_k \left( \psi_k (l_k + l_{k-1}) 
	+ \eta_k (l_k - l_{k-1}) \tan \frac{\theta_k}{2} \right).
\end{align*}
\qed
\end{prf}
\begin{rmk}
For a function $\psi_k$ satisfying $\sum_k \psi_k = 0$, we can find a variation whose variation vector field is $\psi_k N_k$.
The proof is completely the same as in \cite{BadoC1984}.
\end{rmk}
Recall that if $\{ p_k \}_k$ is an equilibrium curve of the functional $L + \kappa \Vol$, then we have $l_k \equiv l_0$, $\theta_k \equiv \theta_0$ and $\kappa l_0 = 2 \tan (\theta_0/2)$.
\begin{lem}
\[
|v_{k+1} - v_k|^2 - \< v_{k+1} - v_k, R \nu_k \>^2 
= [(\psi_{k+1} - \psi_k) + \tan (\theta_0/2) (\eta_{k+1} + \eta_k)]^2.
\]
Therefore we have
\[
|\nabla v_k|^2 - \< \nabla v_k, R \nu_k \>^2
= \left( \nabla \psi_k + \frac{\kappa}{2} (\eta_k + \eta_{k+1}) \right)^2
\]
\end{lem}
\begin{prf}
Recall that
\begin{align*}
\<N_k, -RN_{k+1}\> 
&= \tan \frac{\theta_k}{2} + \tan \frac{\theta_{k+1}}{2} 
= 2 \tan \frac{\theta_0}{2}, \\
\< N_k, N_{k+1} \> 
&= 1 - \tan \frac{\theta_k}{2} \tan \frac{\theta_{k+1}}{2} 
= 1- \tan^2 \frac{\theta_0}{2}.
\end{align*}
If we note $|N_k| = 1/\cos (\theta_0/2)$, then
\begin{align*}
&\qquad |v_{k+1} - v_k|^2 \\
&= |\psi_{k+1} N_{k+1} - \psi_k N_k|^2 
	+ 2 \< \psi_{k+1} N_{k+1} -\psi_k N_k, \eta_{k+1} T_{k+1} - \eta_k T_k \>
		+ |\eta_{k+1} T_{k+1} - \eta_k T_k|^2 \\
&= \frac{\psi_{k+1}^2}{\cos^2(\theta_0/2)} -2 \psi_k \psi_{k+1} (1 - \tan^2 (\theta_0/2)) 
	+ \frac{\psi_k^2}{\cos^2 (\theta_0/2)} 
		- 2 \< N_k, RN_{k+1} \> (\psi_{k+1} \eta_k - \psi_k \eta_{k+1})\\
&\qquad + \frac{\eta_{k+1}^2}{\cos^2(\theta_0/2)} -2 \eta_k \eta_{k+1} (1 - \tan^2 (\theta_0/2)) 
	+ \frac{\eta_k^2}{\cos^2 (\theta_0/2)} \\
&= (1+ \tan^2 (\theta_0/2)) (\psi_{k+1}^2 + \psi_k^2 + \eta_{k+1}^2 + \eta_k^2) \\
&\qquad -2(\psi_k \psi_{k+1} + \eta_k \eta_{k+1})(1 -\tan^2 (\theta_0/2)) 
	+ 4 \tan (\theta_0/2) (\eta_k \psi_{k+1} - \eta_{k+1} \psi_k).
\end{align*}
Similarly since
\begin{align*}
&\qquad \< v_{k+1} - v_k, R \nu_k \> \\
&= \< \psi_{k+1} N_{k+1} - \psi_k N_k + \eta_{k+1} T_{k+1} -\eta_k T_k, R \nu_k \> \\
&= \psi_{k+1} \frac{\< \nu_k - \nu_{k+1}, \nu_k \>}{\sin \theta_{k+1}} 
	-\psi_k \frac{\< \nu_{k-1} - \nu_k, \nu_k \>}{\sin \theta_k} 
		- \eta_{k+1} \frac{\< \nu_{k+1} -\nu_k, R\nu_k \>}{\sin \theta_{k+1}}
			+ \eta_k \frac{\< \nu_k - \nu_{k-1}, R \nu_k \>}{\sin \theta_k} \\
&= (\psi_{k+1} + \psi_k) \tan \frac{\theta_0}{2} - (\eta_{k+1} - \eta_k),
\end{align*}
we have
\[
\< v_{k+1} - v_k, R \nu_k \>^2
= (\psi_{k+1} + \psi_k)^2 \tan^2 \frac{\theta_0}{2} - 2 \tan \frac{\theta_0}{2} (\psi_{k+1} + \psi_k) (\eta_{k+1} - \eta_k) + (\eta_{k+1} - \eta_k)^2.
\]
Substracting these factors we have
\begin{align*}
&\qquad |v_{k+1} - v_k|^2 - \< v_{k+1} - v_k, R \nu_k \>^2 \\
&= \psi_{k+1}^2 + \psi_k^2 + \tan^2 \frac{\theta_0}{2} (\eta_{k+1}^2 + \eta_k^2) 
	-2 \psi_k \psi_{k+1} \tan^2 \frac{\theta_0}{2} + 2 \eta_k \eta_{k+1} \\
&\qquad + 2\tan \frac{\theta_0}{2} (2 \eta_k \psi_{k+1} - 2 \eta_{k+1} \psi_k 
		+ \psi_{k+1} \eta_{k+1} - \psi_{k+1} \eta_k + \psi_k \eta_{k+1} - \psi_k \eta_k) \\
&\qquad \qquad  -2 (\psi_k \psi_{k+1} + \eta_k \eta_{k+1}) 
			+ 2 \tan^2 \frac{\theta_0}{2} (\psi_k \psi_{k+1} + \eta_k \eta_{k+1})\\
&= (\psi_{k+1} - \psi_k)^2 + \tan^2 \frac{\theta_0}{2} (\eta_{k+1} + \eta_k)^2 
	+ 2 \tan \frac{\theta_0}{2} (\psi_{k+1} - \psi_k) (\eta_{k+1} + \eta_k) \\
&= [(\psi_{k+1} - \psi_k) + \tan \frac{\theta_0}{2} (\eta_{k+1} + \eta_k) ]^2
\end{align*}
\qed
\end{prf}
\begin{lem}
\[
\< v_k, R v_{k+1} \> 
= - 2 \tan \frac{\theta_0}{2} (\psi_k \psi_{k+1} + \eta_k \eta_{k+1}) 
	- (1-\tan^2 \frac{\theta_2}{2}) (\eta_k \psi_{k+1} - \eta_{k+1} \psi_k).
\]
\end{lem}
\begin{prf}
This is also a simple calculation:
\begin{align*}
\< v_k, R v_{k+1} \>
&= \< \psi_k N_k - \eta_k R N_k, \psi_{k+1} R N_{k+1} + \eta_{k+1} N_{k+1} \> \\
&= (\psi_k \psi_{k+1} + \eta_k \eta_{k+1}) \< N_k, R N_{k+1} \>
	- (\eta_k \psi_{k+1} - \eta_{k+1} \psi_k) \< N_k, N_{k+1} \> \\
&= - 2 \tan \frac{\theta_0}{2} (\psi_k \psi_{k+1} + \eta_k \eta_{k+1})
	- (1- \tan^2 \frac{\theta_0}{2} )(\eta_k \psi_{k+1} - \eta_{k+1} \psi_k).
\end{align*}
\qed
\end{prf}
\begin{thm}[Second variation formula for the length functional]
\label{thm:SecondVariationCurve}
\begin{align*}
\delta^2 L
= \sum_k \left[ |\nabla \psi_k|^2 - \kappa^2 \psi_k \psi_{k+1} 
	+ \tan^2 \frac{\theta_0}{2} (\kappa \nabla\psi_k (\eta_{k+1} + \eta_k) + |\nabla \eta_k|^2) \right] l_0
\end{align*}
In particular, for the normal variation we have
\[
\delta^2 L
= \sum_k (|\nabla \psi_k|^2 - \kappa^2 \psi_k \psi_{k+1}) l_0
= - \sum_k \psi_k (\Delta \psi_k + \kappa^2 \psi_{k+1} ) l_0,
\]
where we use the integration by parts and take the line element at the vertex as $L_k = (l_k + l_{k-1})/2 = l_0$.
\end{thm}
\begin{prf}
By using the previous lemmas, we have
\begin{align*}
&\qquad (|\nabla v_k|^2 - \< \nabla v_k, R \nu_k \>^2) l_k + \kappa \< v_k, R v_{k+1} \> \\
&= \left( |\nabla \psi_k|^2 + \kappa \nabla \psi_k (\eta_{k+1} + \eta_k) 
	+ \frac{\kappa^2}{4} (\eta_{k+1} + \eta_k)^2 \right) l_0 \\
&\qquad - \kappa^2 l_0 (\psi_k \psi_{k+1} + \eta_k \eta_{k+1})
	- \kappa (1-\tan^2 \frac{\theta_0}{2}) (\eta_k \psi_{k+1} - \eta_{k+1} \psi_k) \\
&= |\nabla \psi_k|^2 l_0 + \kappa \nabla \psi_k (\eta_{k+1} + \eta_k) l_0 
	+ \frac{\kappa^2}{4} (\eta_{k+1} - \eta_k)^2 l_0 - \kappa^2 \psi_k \psi_{k+1} l_0 \\
&\qquad - \kappa (1-\tan^2 \frac{\theta_2}{2}) 
	( (\psi_{k+1} - \psi_k)(\eta_{k+1} + \eta_k) - (\psi_{k+1} \eta_{k+1} - \psi_k \eta_k) ) \\
&= |\nabla \psi_k|^2 l_0 + \kappa \tan^2 \frac{\theta_0}{2} \nabla \psi_k (\eta_{k+1} + \eta_k) l_0 
	+ \frac{\kappa^2}{4} (\eta_{k+1} - \eta_k)^2 l_0 - \kappa^2 \psi_k \psi_{k+1} l_0 \\
&\qquad + \kappa (1 - \tan^2 \frac{\theta_0}{2}) (\psi_{k+1} \eta_{k+1} - \psi_k \eta_k).
\end{align*}
Taking the summation, we have the desired result.
\qed
\end{prf}


\section{Instability of non-convex regular polygons}
\label{sec:InstabilityOfRegularPolygons}

In this section we will prove that non-convex regular polygons and convex regular polygons with multiplicity are unstable.
To prove this, we find a special variation with a help of the following discrete version of Wirtinger's inequality:
\begin{thm}[Discrete Wirtinger's inequality, \cite{FaTaTo1954}]
\label{thm:DiscreteWirtingerInequality}
Let $\psi_0, \ldots, \psi_n$ be $(n+1)$ real numbers such that
\[
\psi_0 = \psi_n, \quad \sum_{k=0}^{n-1} \psi_k = 0.
\]
Then we have
\begin{equation}
\label{eq:2ndVarFormula}
\sum_{k=0}^{n-1} (\psi_{k+1} - \psi_k)^2 \geq 4 \sin^2 \frac{\, \pi \,}{n} \sum_{k=0}^{n-1} \psi_k^2
\end{equation}
and the equality holds if and only if there exist $A, B \in \R$ such that
\[
\psi_k = A \cos \frac{\, 2 \pi k \,}{n} + B \sin \frac{\, 2 \pi k \,}{n}.
\]
\end{thm}
In the following, we will consider the normal variation, i.e., the variation which have the form:
\[
p_k(t) = p_k + t \psi_k N_k + O(t^2), \quad N_k = (\nu_k + \nu_{k-1})/(1 + \cos \theta_0),
\]
where $\psi_k$ satisfies $\sum_k \psi_k = 0$.
By the second variation formula (Theorem \ref{thm:SecondVariationCurve}) we have
\[
\delta^2 L
= \sum_k (|\nabla \psi_k|^2 - \kappa^2 \psi_k \psi_{k+1}) l_0
= \sum_k \frac{\,1 \,}{l_0} \left[ (\psi_{k+1} - \psi_k)^2 - 4 \psi_k \psi_{k+1} \tan^2 \frac{\, m \pi \,}{n} \right]
\]
for any admissible variations, where we use the relation $\kappa l_0 = 2 \tan (\theta_0/2)$ and put $\theta_0 = 2 m \pi/ n$ for some $m \in \Z$ and assume that $m/n \neq 1/2$.
\begin{thm}[Instability of non-convex regular polygons]
\label{thm:InstabilityOfRegularPolygons}
Let $n \geq 5$.
By taking $\psi_k = A \cos (2 \pi k/n) + B \sin (2 \pi k/n)$, $(A, B) \neq (0, 0)$, we have
\[
\delta^2 \Length
= \frac{\, 4 \,}{l_0} \left[ \sin^2 \frac{\, \pi \,}{n} - \cos \frac{\, 2 \pi \,}{n} \tan^2 \frac{\, m \pi \,}{n} \right] \sum_k \psi_k^2.
\]
In particular, $\delta^2 \Length < 0$ for $2 \leq m \leq n-2$, i.e., non-convex regular polygons are unstable.
\end{thm}
\begin{proof}
By the discrete Wirtinger's inequality we have
\[
\delta^2 \Length \geq \sum_k \frac{\, 4 \,}{l_0} (\psi_k^2 \sin^2 \frac{\pi}{n} - \psi_k \psi_{k+1} \tan^2 \frac{m \pi}{ n}).
\]
In the following we use the equality condition $\psi_k = A \cos (2 \pi k/n) + B \sin (2 \pi k/n)$ and put $\vp_k = -A \sin (2 \pi k/n) + B \cos (2\pi k/n)$.
Then we have
\[
\psi_{k+1} = \psi_k \cos(2 \pi/n) +\vp_k \sin (2 \pi /n),\quad
\psi_k \vp_k = \frac{\,1\,}{2} (B^2 - A^2) \sin \frac{4 k \pi }{n} + AB \cos \frac{4 k \pi }{n}.
\]
If we note the fact $\sum_k \psi_k \vp_k = 0$, then
\begin{align*}
\delta^2 \Length
&= \sum_k \frac{\, 4 \,}{l_0} \left[ \psi_k^2 \sin^2 \frac{\, \pi \,}{n} - \psi_k \left( \psi_k \cos \frac{2 \pi}{n} +\vp_k \sin \frac{2 \pi}{n} \right) \tan^2 \frac{m \pi}{n} \right] \\
&= \frac{\, 4 \,}{l_0} \left[ \sin^2 \frac{\, \pi \,}{n} - \cos \frac{2 \pi}{n} \tan^2 \frac{m \pi}{n} \right] \sum_k \psi_k^2.
\end{align*}
If $m=1$ or $m=n-1$, then
\[
\delta^2 \Length \geq \frac{\, 4 \,}{l_0} \sin^2 \frac{\, \pi \,}{n} \tan^2 \frac{\, \pi \,}{n} \sum_k \psi_k^2 \geq 0.
\]
On the other hand, for $2 \leq m \leq n-2$ and $(A, B) \neq (0, 0)$ we have
\[
\delta^2 \Length
\leq - \frac{\, 4 \sin^2 (\pi/n) (1 + 2 \cos^2 (\pi/n)) \,}{l_0 \cos (2 \pi/n)} \sum_k \psi_k^2 < 0,
\]
where we use the fact $\tan^2 (m \pi/n) \geq \tan^2 (2 \pi/ n)$ if $2 \leq m \leq n-2$, and $\cos (2\pi /n) > 0$ if $n \geq 5$.
This proves the statement.
\end{proof}


\section{Appendix}

We observe the second variation formula from the analysis of the Jacobi operator.
We can modify the equation \eqref{eq:2ndVarFormula} as follows:
\begin{align*}
\delta^2 L
= \frac{\,1\,}{l_0} \sum_k (- \alpha \psi_{k-1} + 2 \psi_k - \alpha \psi_{k+1}) \psi_k 
= \frac{\,1\,}{l_0} \<H \Psi, \Psi \>,
\end{align*}
where we put $\alpha = 1 + 2 \tan^2 (m\pi/n)$ and
\[
H =
\begin{pmatrix}
2 & -\a & 0 & \cdots & 0 & -\a \\
-\a & 2 & -\a & \cdots & 0 & 0 \\
0 & -\a & 2 & \cdots & 0 & 0 \\
\vdots & \vdots & \vdots & \ddots  & \vdots & \vdots \\
0 & 0 & 0 & \cdots & 2 & -\a \\
-\a & 0 & 0 & \cdots & -\a & 2 \\
\end{pmatrix},\quad
\Psi =
\begin{pmatrix}
\psi_1 \\
\psi_2 \\
\psi_3 \\
\vdots \\
\psi_n
\end{pmatrix}.
\]
Since the matrix $H$ is the circulant matrix, the eigenvalues $\lambda_j$ can be calculated explicitly, see e.g. \cite{Dav1979}:
\[
\lambda_j 
= 2 - 2 \a \cos \frac{2\pi j}{n} 
= \frac{4 \cos^2 (j \pi/n)}{\cos^2 (m\pi/n)} 
	\left( \tan^2 \frac{j \pi}{n} - \sin^2 \frac{m \pi}{n} \right),
		\quad j = 1, \ldots, n.
\]
The corresponding eigenvectors are
\begin{equation}
\label{eq:Eigenvectors}
e_j = {}^t (1, \omega^j, \ldots, \omega^{j(n-1)}) \in \C^n, \quad
\omega = \exp(2 \pi \sqrt{-1}/n).
\end{equation}
The condition $\sum_k \psi_k = 0$ is equivalent to the condition that $\Psi = {}^t (\psi_1, \ldots, \psi_n)$ is perpendicular to $e_n = {}^t (1, \ldots, 1)$.
Therefore we only consider the eigenvalues $\lambda_1, \ldots, \lambda_{n-1}$.
\begin{lem}
For $m \leq j \leq n-m$, we have $\lambda_j > 0$.
In particular, $\lambda_j > 0$ for $1 \leq j \leq n-1$ if $m=1$, i.e., the convex regular polygon case.
\end{lem}
\begin{prf}
This is a direct calculation:
\[
\tan^2 \frac{j \pi}{n} - \sin^2 \frac{m \pi}{n}
\geq \tan^2 \frac{m \pi}{n} - \sin^2 \frac{m \pi}{n}
= \sin^2 \frac{m \pi}{n} \tan^2 \frac{m \pi}{n} > 0.
\]
\end{prf}
\begin{thm}
\label{thm:InstabilityOfRegularPolygons}
Assume $m \geq  2$.
For $1 \leq j \leq m/2$ or $n - m/2 \leq j \leq n-1$, we have $\lambda_j < 0$.
Therefore, the index of a non-convex regular polygon or a convex polygon with multiplicity is at least $\lfloor m/2 \rfloor$.
\end{thm}
\begin{prf}
Under the above conditions, we have
\[
\tan^2 \frac{j \pi}{n} - \sin^2 \frac{m \pi}{n}
\leq \tan^2 \frac{j \pi}{n} - \sin^2 \frac{2 j \pi}{n} < 0.
\]
\qed
\end{prf}
\begin{rmk}
More precisely, $\lambda_j < 0$ for $j < (n /\pi) \arctan (\sin (m\pi/n))$.
\end{rmk}
\begin{rmk}
From another point of view, a non-convex regular polygon is the high-frequency component of the discrete Fourier expansion of the polygon.
More precisely, any polygon in $\R^2$ with $n$-vertices can be regarded as a point in $\R^{2n} \simeq \C^n$.
Each eigenvector $e_k$ in \eqref{eq:Eigenvectors} corresponds to the regular $n$-gon and $\{ e_1, \ldots, e_n \}$ forms a basis of $\C^n$. 
Therefore, any polygon in $\R^2$ can be written as a linear combination of the regular $n$-gons and this fact  corresponds to the discrete Fourier expansion.
\end{rmk}


\section*{Acknowledgements}
The author would like to experess his gratitude to Professor Miyuki Koiso and Professor Hisashi Naito for invaluable comments and friutful discussions.


\begin{flushleft}
Yoshiki J{\footnotesize IKUMARU} \\
Institute of Mathematics for Industry, Kyushu University  \\
744 Motooka Nishi-ku, Fukuoka 819-0395, Japan \\
E-mail: {\tt y-jikumaru@imi.kyushu-u.ac.jp} \\
 \end{flushleft}
\end{document}